\documentclass[10pt]{amsart}
\usepackage{latexsym,amssymb}
\usepackage{tikz}
\theoremstyle{plain}
  \newtheorem{theorem}{Theorem}[section]
  \newtheorem{proposition}[theorem]{Proposition}
  
  \newtheorem{corollary}[theorem]{Corollary}
  
\theoremstyle{definition}
  \newtheorem{definition}[theorem]{Definition}
  \newtheorem{example}[theorem]{Example}

 \theoremstyle{remark}
  \newtheorem{remark}[theorem]{Remark}

\newcommand\qtbin[2]{\left[\begin{matrix} #1 \\ #2 \end{matrix} \right]}

\newcommand\Hilb{\operatorname{Hilb}}
\newcommand\Hom{\operatorname{Hom}}

\newcommand\inv{\operatorname{inv}}
\newcommand\maj{\operatorname{maj}}

\newcommand\wt{\operatorname{wt}}

\newcommand\Des{\operatorname{Des}}
\newcommand\specialize{\operatorname{sp}}

\newcommand\Symm{\mathfrak{S}}

\newcommand\xx{{\mathbf{x}}}

\newcommand\FFF{\mathbf{F}}

\newcommand\LLL{{\mathcal L}}

\newcommand\FQSym{{\mathcal{FQS}ym}}
\newcommand\PBT{{\mathcal{PBT}}}

\newcommand\ZZ{{\mathbb Z}}
\newcommand\FF{{\mathbb F}}
\newcommand\QQ{{\mathbb Q}}
\newcommand\NN{{\mathbb N}}

\usepackage{ifthen}
\newboolean{draft}
\setboolean{draft}{true}
\ifdraft
\newcommand{\TODO}[2][To do: ]{\textcolor{red}{\textbf{#1#2}}}
\else
\newcommand{\TODO}[2][]{}
\fi

\numberwithin{equation}{section}

\title[A multivariate ``inv'' hook formula for forests]
{A Multivariate ``inv'' hook formula for forests}

\dedicatory{To Dennis Stanton on his 60th birthday}

\author{Florent Hivert}
\email{florent.hivert@univ-rouen.fr}
\address{LITIS, Universit\'e de Rouen\\
Avenue de l'Universit'e\\
76801 Saint \'Etienne du Rouvray\\
France}

\author{Victor Reiner}
\email{reiner@math.umn.edu}
\address{School of Mathematics\\
University of  Minnesota\\
Minneapolis, MN 55455\\
USA}

\thanks{}
\thanks{First author partially supported by grant ANR-06-BLAN-0380.
Second author partially supported by NSF grant DMS-0601010.
The second author also thanks A. Lascoux, J.-C. Novelli, and J.-Y. Thibon
of the Institut Gaspard Monge at the University of Marne-la-Vall\'ee
for their hospitality during part of this work.}



\begin{document}

\begin{abstract}
Bj\"orner and Wachs provided two $q$-generalizations of 
Knuth's hook formula counting linear extensions of forests: one
involving the major index statistic, and one involving the inversion
number statistic.  We prove a multivariate generalization of
their inversion number result, motivated by specializations
related to the modular invariant theory of finite general linear groups.
\end{abstract}

\maketitle

%\tableofcontents

%%%%%%%%%%%%%%%%%%%%%%%%%%%%%%%%%%%%%
\section{Introduction}
\label{intro-section}
%%%%%%%%%%%%%%%%%%%%%%%%%%%%%%%%%%%%%

This paper concerns formulas counting linear extensions
of partial orders $P$ on the set $\{1,2,\ldots,n\}$ which are {\it forests}, in the
sense that every element covers at most one other element.
Recall that a permutation $w$ is a {\it linear extension} 
of the poset $P$ if the linear order
$w_1 <_w \ldots <_w w_n$ has the property that $i <_P j$ implies
$i <_w j$.  Denote by $\LLL(P)$ the set of all linear extensions
of $P$.  Knuth observed the following.
\vskip.1in
\noindent
{\bf Theorem.} (Knuth \cite[\S 5.1.4, Exer. 20]{Knuth})
{\it
For any forest poset $P$ on $\{1,2,\ldots,n\}$, one has
$$
| \LLL(P) | = \frac{n!}{\prod_{i=1}^n h_i}
$$
where $h_i:=|P_{\geq i}|$ is the cardinality of the subtree
$P_{\geq i}$ rooted at $i$.
}
\vskip.1in
\noindent
Bj\"orner and Wachs \cite{BjornerWachs} later gave two interesting
$q$-generalizations of Knuth's result, both counting linear extensions according
to certain statistics: the {\it inversion number} statistic $\inv$,
and the {\it major index} statistic $\maj$.
The following theorem rephrases a special case of the
first of these results, relating to $\inv$; see
Remark~\ref{maj-formula-remark} below for their second generalization.

Say that a forest poset $P$ is {\it recursively labelled} if the label set on
each subtree $P_{\geq i}$ forms an interval in the integers, that is, $P_{\geq
  i}=\{a,a+1,\ldots,b-1,b\}$ for some integers $a=:\min(P_{\geq i})$ and
$b=:\max(P_{\geq i})$.  Define the {\it inversion number} $\inv(P)$ to be the
number of pairs $i <_\ZZ j$ for which $i >_P j$. For example, the following
picture shows the Hasse diagram of a recursively labelled forest $P$
on $\{1,2,\ldots,10\}$.
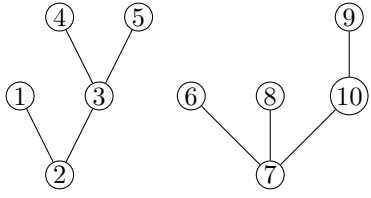
\[
\begin{tikzpicture}[scale=0.7, inner sep=0.3mm
%,edge from parent/.style={<-,draw}
]
  \node[circle, draw] {2} [grow'=up]
    child {node[circle, draw] {1}}
    child {node[circle, draw] {3}
      child {node[circle, draw] {4}}
      child {node[circle, draw] {5}}
    };
  \node[circle, draw] at (4,0) {7} [grow'=up]
    child {node[circle, draw] {6}}
    child {node[circle, draw] {8}}
    child {node[circle, draw] {10}
      child {node[circle, draw] {9}}
    };
\end{tikzpicture}
\]
Here one has $P_{\geq 3} = \{3,4,5\}, P_{\geq 7} = \{6,7,8,9,10\}$, and
$$
\inv(P)=3=|\{(1,2), (6,7),(9,10)\}|.
$$

Lastly, define the
$q$-analogues
$$
\begin{aligned}[]
[n]_q&:=1+q+q^2+\cdots+q^{n-1},\\
[n]!_q &:=[n]_q [n-1]_q [n-2]_q \cdots [2]_q [1]_q.
\end{aligned}
$$
\vskip.1in
\noindent
{\bf Theorem.} (Bj\"orner and Wachs \cite[Thm. 1.1]{BjornerWachs})
~\\
{\it
Any recursively labelled forest $P$ on $\{1,2,\ldots,n\}$ has
\begin{equation}
\label{Bjorner-Wachs-equation}
\sum_{w \in \LLL(P) } q^{\inv(w)} = q^{\inv(P)} \frac{[n]!_q}{\prod_{i=1}^n [h_i]_q}.
\end{equation}
}
\vskip.1in

Our goal is a multivariate generalization, Theorem~\ref{main-theorem}
below.  It is an identity within the field of rational functions
$\QQ(\xx):=\QQ(x_1,x_2,x_3,\ldots)$ in a sequence of indeterminates 
$x_1,x_2,x_3,\ldots$, related by a map $F$ sending $x_i \mapsto x_{i+1}$
that we call the {\it Frobenius map}.  
We introduce the following multivariate analogues of 
the positive integers $n$ and the factorial $n!$:
\begin{align}
  \label{n-definition}
[1]  :=&x_1\\\notag
[n]  :=&\ [1]+F[1]+F^2[1]+\cdots +F^{n-1}[1]\\\notag
      =&\ x_1 + x_2 + \cdots + x_n\\[3mm]
  \label{n!-definition}
[n]! :=&\ [n] \cdot F([n-1]) \cdot F^2([n-2]) \cdots F^{n-2}([2]) \cdot F^{n-1}([1])\\\notag
      =&\ [n] \cdot F\left( [n-1]! \right).
\end{align}
For example,
$$
[4]!=(x_1+x_2+x_3+x_4)(x_2+x_3+x_4)(x_3+x_4)x_4.
$$
After defining in Section~\ref{permutation-section} a weight $\wt(w)$
lying in $\QQ(\xx)$ for each permutation $w$, we prove in
Section~\ref{proof-section} the following main result.

\begin{theorem}
\label{main-theorem}
Any recursively labelled forest $P$ on $\{1,2,\ldots,n\}$ has
$$
L(P)\ :=\ \sum_{w \in \LLL(P) } \wt(w)
\ =\ \frac{[n]!}{\prod_{i=1}^n F^{\min(P_{\geq i})-1}[h_i]}\,.
$$
\end{theorem}

\noindent
Section~\ref{corollaries-section} explains why
Theorem~\ref{main-theorem} becomes 
\eqref{Bjorner-Wachs-equation} upon applying the
following $q$-{\it specialization} map to both sides:

\begin{equation}
\label{q-specialization-map}
\begin{array}{rcl}
\QQ(x_1,x_2,\ldots) &\overset{\specialize_q}{\longrightarrow} &\QQ(q) \\
x_i &\longmapsto &q^{i-1} - q^i\,.
\end{array}
\end{equation}

%%%%%%%%%%%%%%%%%%%%%%%%%%%%%%%%%%%%%
\section{Invariant theory motivation}
\label{invariant-theory-section}
%%%%%%%%%%%%%%%%%%%%%%%%%%%%%%%%%%%%%

Aside from the Bj\"orner-Wachs {\it inv} formula,
a second motivation for Theorem~\ref{main-theorem} stems
from previous joint work in invariant theory with D. Stanton \cite{RStanton}.
The reader interested mainly in Theorem 1.1 and
its connection to the work of Bj\"orner and Wachs can safely
skip this explanation of the invariant-theoretic connection.

There are two special cases of Theorem~\ref{main-theorem} that turn out be equivalent to 
results from \cite{RStanton}, namely the cases where either
\begin{enumerate}
\item[(a)]
$P$ is a disjoint union of chains, 
each labelled by a contiguous interval of integers in increasing order \cite[Theorem 8.6]{RStanton},
or 
\item[(b)]
$P$ is a {\it hook} poset \cite[Eqn. (6.1) and (11.1)]{RStanton}, having 
$$ 
1 >_P 2 >_P \cdots >_P m <_P m+1 <_P \cdots <_P n-1 <_P n\,.
$$
\end{enumerate}

The story from \cite{RStanton} begins with $G:=GL_n(\FF_q)$ acting by linear substitutions
of variables on the polynomial algebra
$
S(q):=\FF_q[x_1,\ldots,x_n].
$
A well-known result of L.E. Dickson asserts that the $G$-invariant subalgebra 
$S(q)^G$ is again a polynomial algebra.

For each composition $\alpha=(\alpha_1,\ldots,\alpha_\ell)$
of $n$, one associates two families of $G$-representations $V(q)$ over $\FF_q$,
described below.  For both of these representations $V(q)$, the graded intertwiner spaces 
$$
M(q):=\Hom_{\FF_q G}(V(q),S(q))
$$
were shown in \cite{RStanton} to be {\it free} modules over $S(q)^G$,
and explicit formulas were given for the degrees of their
$S(q)^G$-basis elements, or equivalently for the {\it Hilbert
series} 
$$
\Hilb_q(t):=\Hilb\left(\,M(q)/S(q)^G_+ M(q)\,,\,\, t\,\right).
$$
These Hilbert series come from generating functions in $\QQ(\xx)$ by
applying the following {\it $(q,t)$-specialization map}
\begin{equation}
\label{q-t-specialization-map}
\begin{array}{rcl}
\QQ(x_1,x_2,\ldots) & \overset{\specialize_{q,t}}{\longrightarrow} &\QQ(q) \\
x_i & \longmapsto & t^{q^{i-1}} - t^{q^i}
\end{array}
\end{equation}
which is less drastic than the specialization in \eqref{q-specialization-map}.

The first family of $G$-representations $V(q)$ associated to $\alpha$
is the permutation module for $G$ acting on $\alpha$-flags of $\FF_q$-subspaces
$$
0 \subset V_{\alpha_1} \subset
V_{\alpha_1 + \alpha_2} \subset
V_{\alpha_1 + \alpha_2+\alpha_3} \subset
\cdots \subset \FF_q^n
$$
where $\dim_{\FF_q} V_i = i$.
For this family one has $\Hilb_q(t) = \specialize_{q,t}L(P)$
where the poset $P$ is as described in case (a)
above, when the chains have lengths $\alpha_1,\ldots,\alpha_\ell$.

The second family of $G$-representations $V(q)$ associated
to $\alpha$ is the homology with $\FF_q$-coefficients
of the subcomplex of the {\it Tits building} generated by the
faces indexed by $\alpha$-flags. 
For this family one has $\Hilb_q(t) = \specialize_{q,t}L(P)$
where the poset $P$ is the {\it rim hook} poset $P$ for $\alpha$, having increasing 
chains of lengths $\alpha_1,\ldots,\alpha_\ell$, generalizing
the $\alpha=(1^m,n-m)$ case described in (b) above.

In fact, for either of these classes of posets $P$ associated to $\alpha$, the 
more drastic $q$-specialization $\specialize_q L(P)$
was shown to have two parallel representation-theoretic and invariant-theoretic 
interpretations.  On one hand, $\specialize_q L(P)=\dim_{\FF_q} V(q)$.  
On the other hand, both classes of $\FF_qG$-modules $V(q)$ have ($q=1$)
analogous $\ZZ W$-module counterparts $V$ where $W=\Symm_n$ is the
symmetric group.  In particular, when one regards $W$ acting on $S:=\ZZ[x_1,\ldots,x_n]$
by permuting the variables, so that $S^W$ is the ring of symmetric polynomials,
one finds that the graded intintertwiner space 
$$
M:=\Hom_{\ZZ W}( V , S )
$$ 
turns out to be a free $S^W$-module, and that
$$
\Hilb(M/S^W_+ M,q)=\specialize_q L(P).
$$

%%%%%%%%%%%%%%%%%%%%%%%%%%%%%%%%%%%%%
\section{Binomial coefficient and Pascal recurrence}
\label{Pascal-section}
%%%%%%%%%%%%%%%%%%%%%%%%%%%%%%%%%%%%%

\begin{definition}(cf. \cite[(1.2)]{RStanton})
\label{binomial-definition}
Define
%\footnote{The invariant-theoretic interpretation
%(see \cite[??]{RStantonWhite} or \cite[??]{RStanton} for 
%this binomial coefficient is that its image under the $q-t$-specialization
%map $\specialize_{q,t}$ gives the degrees of the 
%basis elements for the $\FF_q[\xx]^P$-invariant
%subalgebra as a free $\FF_q[\xx]^G$-module, where
%$P$ is the parabolic subgroup of $G$ stabilizing a particular
%$k$-plane in $\FF_q^n$.}
a multivariate analogue of a binomial coefficient
$$
\qtbin{n}{k}:= \frac{[n]!}{[k]! \cdot F^k( [n-k]! )}.
$$
\end{definition}

\noindent
It is an easy exercise in the definitions \eqref{n!-definition}
to deduce the following analogue of the usual Pascal recurrence.

\begin{proposition}\textup{(\textit{cf.} \cite[1st equation in (4.2)]{RStanton})}
\label{Pascal-proposition}
$$
\qtbin{n}{k} = F \qtbin{n-1}{k-1} + 
                    \frac{F[k]!}{[k]!} \cdot F \qtbin{n-1}{k}. \qed
$$
\end{proposition}

%%%%%%%%%%%%%%%%%%%%%%%%%%%%%%%%%%%%%
\section{The weight of a subset}
\label{partition-section}
%%%%%%%%%%%%%%%%%%%%%%%%%%%%%%%%%%%%%

The Pascal recurrence leads to an interpretation of the binomial coefficient
as a sum over certain partitions (cf. \cite[(5.1)]{RStanton}). For our purpose, it is better
to rephrase it as weight $\wt(S)$ defined for sets $S$ of positive integers: 
a $k$-element set 
\begin{equation}
\label{typical-subset}
S=\{i_1>i_2> \dots > i_k\}
\end{equation}
of positive integers, indexed in decreasing order, bijects
with a partition $\lambda$ whose Ferrers diagram fits inside a $k\times(n-k)$ rectangle:
\begin{equation}
\label{bijection-partition-set}
  \lambda(S) := (i_1,i_2, \dots, i_k) - (k, k-1, \dots, 2,1)\,.
\end{equation}
We thus re-encode the definition in \cite[(5.1)]{RStanton} as follows.
\begin{definition}
For a $k$-element set $S$ of positive integers indexed as in \eqref{typical-subset},
define
$$
\wt(S) := \frac{\prod_{j=1}^k F^{i_j-1}[j]}{[k]!}
        = \prod_{j=1}^k \frac{F^{i_j-1}[j]}{F^{k-j}[j]}.
$$

\begin{example}
  For $k=5$, the set $S = \{9,7,6,4,2\}$ has weight
  \begin{equation*}
    \wt(S) =  \frac{F^8[1] F^6[2] F^5[3] F^3[4] F^1[5]}{[5]!}.
  \end{equation*}
\end{example}

Using the notation 
$$
S+1:=\{i+1: i \in S\}
$$
one can also define this weight recursively as follows:
\begin{equation}
\label{recursive-lambda-weight-definition}
\wt(S):=
\begin{cases}
1 & \text{ if }S=\emptyset\\
\frac{F[k]!}{[k]!} F \wt(\hat{S}) & \text{ if }1 \not\in S \text{ and }S=\hat{S}+1 \\
F \wt(\hat{S}) & \text{ if }1 \in S\text{ and }S  = \{1\} \cup \left( \hat{S}+1 \right).
\end{cases}
\end{equation}
\end{definition}

\begin{proposition}\textup{(\textit{cf.} \cite[Theorem 5.3]{RStanton})}
\label{binomial-as-sum}
$$
\qtbin{n}{k} = \sum_S \wt(S)
$$
where the sum runs over all subsets $S$ of cardinality $k$ of $\{1,\dots n\}$.
\end{proposition}
\begin{proof}
Induct on $n$ with trivial base case $n=0$.  
In the inductive step, the sum in the right hand side of the proposition
decomposes as two subsums
$$
\sum_{1\in S} \wt(S) 
+\sum_{1\notin S} \wt(S)
$$
which correspond to the two terms in the 
Pascal recurrence, Proposition~\ref{Pascal-proposition}.
Using the recursive definition \eqref{recursive-lambda-weight-definition} then
completes the inductive step.
\end{proof}

%%%%%%%%%%%%%%%%%%%%%%%%%%%%%%%%%%%%%
\section{The weight of a permutation via recursion}
\label{permutation-section}
%%%%%%%%%%%%%%%%%%%%%%%%%%%%%%%%%%%%%

We wish to extend the definition of the weight $\wt(S)$ for a set $S$
to a weight $\wt(w)$ for permutations $w$ in $\Symm_n$, defined
recursively, following \cite[\S 8]{RStanton}.

\begin{definition}\cite[Definition 8.1]{RStanton}
\label{definition-weight-permutation}
Given $w=(w_1,w_2,\ldots,w_n)$ in $W:=\Symm_n$, let $k:=w_1-1$, so
that $0 \leq k \leq n-1$ and $w_1=k+1$.  Regarding $w$ as
a shuffle of its restrictions to the alphabets $[1,k]$ and
$[k+1,n]$, one can factor it uniquely 
\begin{equation}
\label{parabolic-factorization}
w=u \cdot a \cdot b
\end{equation}
with $u$ a minimum-length coset representative of $uW_J$
for the {\it parabolic} or {\it Young subgroup}
$$
\begin{aligned}
W_J&:=\Symm_{[1,k]} \times \Symm_{[k+1,n]}\\
&\cong \Symm_k \times \Symm_{n-k}
\end{aligned}
$$
and where $a, b$ lie in $\in \Symm_{[1,k]}, \Symm_{[k+1,n]}$, respectively.

Since $u$ is a shuffle of the increasing sequences $(1,2,\ldots,k), (k+1,k+2,\ldots,n)$, 
it can be encoded via the set
\begin{equation}
\label{Grassmannian-to-set-encoding}
S(u) := \{u^{-1}(k)>u^{-1}(k-1)>\cdots>u^{-1}(2)>u^{-1}(1)\}\,.
\end{equation}

Since $w_1=k+1$ implies $b(k+1)=k+1$,
the permutation $b$ in $\Symm_{[k+1,n]}$
actually lies in the subgroup $\Symm_{[k+2,n]}$ that fixes $k+1$,
isomorphic to $\Symm_{n-k-2}$.  Denote by $\hat{b}$
the corresponding element of $\Symm_{n-k-2}$.
%The same fact also implies that $\lambda_k \geq 0$,
%and hence the partition $\hat(\lambda)=(\lambda_1-1,\ldots,\lambda_k-1)$
%is well-defined and satisfies $\wt(\lambda;k) = \frac{F[k]!}{[k]!} F\wt(\lambda;k)$
%by \eqref{recursive-lambda-weight-definition}.

Now define $\wt(w)$ recursively by saying
that the identity element $e$ in $\Symm_0$ has
$\wt(e):=1$, and otherwise
\begin{equation}
\label{w-weight-definition}
\wt(w):= \wt(S(u)) \cdot \wt(a) \cdot F^{k+1}( \wt(\hat{b}) ).
\end{equation}
\end{definition}

\noindent
Note that since $k = w_1-1$, the integer $1$ is never in $S(u)$. Therefore
writing $S(u)= \hat{S}(u)+1$ for a $k$-element subset of $\{1,2,\ldots,n-1\}$, 
one can use
\eqref{recursive-lambda-weight-definition} to rewrite
\eqref{w-weight-definition} as
\begin{equation}
\label{w-weight-with-hat-lambda}
\wt(w):=  \frac{F[k]!}{[k]!} F\left(\wt(\hat{S}(u)) \right)
               \cdot \wt(a) \cdot F^{k+1}( \wt(\hat{b}) ).
\end{equation}

\begin{example}
For $n=9$, consider within $\Symm_9$ the permutation
$$
\begin{aligned}
w&=
\left(
\begin{matrix}
1 & \underline{2} & 3 & \underline{4} & 5 & \underline{6} & \underline{7} & 8 & \underline{9}\\
6 & \mathbf{2} & 9 & \mathbf{1} & 7 & \mathbf{5} & \mathbf{3} & 8 & \mathbf{4}
\end{matrix}
\right) \\
&=
\underbrace{\left(
\begin{matrix}
1 & \underline{2} & 3 & \underline{4} & 5 &  \underline{6} & \underline{7} & 8 & \underline{9}\\
6 & \mathbf{1} & 7 & \mathbf{2} & 8 & \mathbf{3} & \mathbf{4} & 9 & \mathbf{5}
\end{matrix}
\right)}_u
\cdot
\underbrace{\left(
\begin{matrix}
1 & 2 & 3 & 4 & 5\\
2 & 1 & 5 & 3 & 4
\end{matrix}
\right)}_{a}
\cdot
\underbrace{\left(
\begin{matrix}
6 & 7 & 8 & 9\\
6 & 9 & 7 & 8
\end{matrix}
\right)}_{b}\\
\end{aligned}
$$
One has $k=w_1-1=6-1=5$ here, and note that $b(6)=6$, with 
$$
\hat{b}=
\left(
\begin{matrix}
1 & 2 & 3\\
3 & 1 & 2
\end{matrix}
\right).
$$
Since the values $\{1,2,3,4,5(=k)\}$ occur in positions
$S(u)=\{9,7,6,4,2\}$ of $u$ or $w$, one has that 
$$
\begin{aligned}
\wt(w) 
&= \wt(\{9,7,6,4,2\}) \cdot \wt(a) \cdot F^6 \wt(\hat{b}) \\
&= \frac{F^8[1] F^6[2] F^5[3] F^3[4] F[5]}
          {[5]!} \cdot \wt(a) \cdot F^6 \wt(\hat{b}).
%         &= \frac{F[5]!}{[5]!} F\wt((3,2,2,1,0);5) 
%                  \cdot \wt(a) \cdot F^6 \wt(\hat{b})
\end{aligned}
$$
Finishing the recursive computation, one finds
\begin{equation*}
  \wt(a)=\wt(21534) = 
  \frac{{\left(x_{4} + x_{5}\right)} x_{2} x_{5}}
  {{\left(x_{3} + x_{4}\right)} x_{1} x_{4}}\,,
  \qquad\qquad
  \wt(b)=\wt(312) =
  \frac{{\left(x_{2} + x_{3}\right)} x_{3}}
  {{\left(x_{1} + x_{2}\right)} x_{2}}\,,
\end{equation*}
\begin{equation*}
  \wt(\{9,7,6,4,2\}) =
  \frac{ x_{9}
    {\left(x_{7} + x_{8}\right)}
    {\left(x_{6} + x_{7} + x_{8}\right)}
    {\left(x_{4} + x_{5} + x_{6} + x_{7}\right)}
    {\left(x_{2} + \dots + x_{6}\right)}}
  { x_{5}
    {\left(x_{4} + x_{5}\right)}
    {\left(x_{3} + x_{4} + x_{5}\right)}
    {\left(x_{2} + x_{3} + x_{4} + x_{5}\right)}
    {\left(x_{1} + \dots + x_{5}\right)}}\,,
\end{equation*}
and therefore
\begin{equation*}
  \wt(w) =
  \frac{ x_{2} x_{9}^{2}
    {\left(x_{8} + x_{9}\right)}
    {\left(x_{6} + x_{7} + x_{8}\right)}
    {\left(x_{4} + x_{5} + x_{6} + x_{7}\right)}
    {\left(x_{2} + \dots + x_{6}\right)}}
  { x_{1} x_{4} x_{8}
    {\left(x_{3} + x_{4}\right)}
    {\left(x_{3} + x_{4} + x_{5}\right)}
    {\left(x_{2} + x_{3} + x_{4} + x_{5}\right)}
    {\left(x_{1} + \dots + x_{5}\right)}}\,.
\end{equation*}
\end{example}

\begin{example}
Here are the values of $\wt(w)$ for $w$ in $\Symm_3$:
\def\arraystretch{1.8}
\begin{equation*}
\begin{array}{|c|c|}\hline
w & \wt(w)
 \\\hline \hline
123 & 1 \\\hline
132 & \frac{F^2[1]}{F[1]}=\frac{x_3}{x_2} \\\hline
213 & \frac{F[1]}{[1]}=\frac{x_2}{x_1}    \\\hline
231 & \frac{F^2[1]}{[1]}=\frac{x_3}{x_1}  \\\hline
312 & \frac{F[2]!}{[2]!}=\frac{F[2]F^2[1]}{[2]F[1]}
        =\frac{(x_2+x_3)x_3}{(x_1+x_2)x_2}    \\\hline
321 & \frac{F[2]!}{[2]!}\frac{F[1]}{[1]}
        =\frac{F[2]F^2[1]}{[2][1]}
        =\frac{(x_2+x_3)x_3}{(x_1+x_2)x_1}\\\hline
\end{array}
\end{equation*}
\vskip.1in
\noindent
Four out of these six permutations $w$  in $\Symm_3$,
namely all except for $\{213, 231\}$,
are themselves recursively labelled forests when regarded as linear orders.
For these four one can check that the value of $\wt(w)$
given in the table agrees with the product formula predicted
by Theorem~\ref{main-theorem}.

On the other hand, the two exceptions $\{213, 231\}$ comprise
$\LLL(P)$ for the recursively labelled forest poset $1 >_P 2 <_P 3$.
One then checks from the values in the table that
$$
L(P)=\wt(213)+\wt(231) =
\frac{x_2}{x_1}+\frac{x_3}{x_1}
=\frac{x_2+x_3}{x_1}=\frac{F[2]}{[1]}
$$  
which again agrees with the prediction of
Theorem~\ref{main-theorem}, namely
$$
\frac{[3]!}{F^{\min(P_\geq 1)-1}[h_1] \cdot
             F^{\min(P_\geq 2)-1}[h_2] \cdot
              F^{\min(P_\geq 3)-1}[h_3]}
= \frac{[3]F[2]F^2[1]}{F^0[1]F^0[3]F^2[1]}
= \frac{F[2]}{[1]}.
$$
\end{example}

\vskip.2in

For later use in Section~\ref{corollaries-section},
we explain how $\wt(w)$ behaves under the 
specialization map $\specialize_q$ from 
\eqref{q-specialization-map} which sends $x_i=F^i[1]$ to $q^{i-1}-q^i$.
Note that 
$$
\specialize_q F^i[n]  = \specialize_q (x_i+x_{i+1} + \cdots + x_{i+n-1})
                      =q^{i-1} - q^{i+n-1}
$$
so that  
\begin{equation}
\label{quotient-q-specialized}
\specialize_q \frac{F^a[n]}{F^b[m]} 
= q^{a-b} \frac{1-q^n}{1-q^m}.
\end{equation}
In particular, when $m=n$ one has 
\begin{equation}
\label{each-Frobenius-is-a-q}
\specialize_q \frac{F^a[n]}{F^b[n]}  = q^{a-b},
\end{equation}
and hence for a $k$-subset 
$S = \{i_1>i_2> \dots > i_k\}$,
\begin{equation}
\label{specialize-set-weight}
\specialize_q \wt(S)
= \specialize_q \prod_{j=1}^k \frac{F^{i_j-1}[j]}{F^{k-j}[j]} 
= q^{\sum_{j=1}^k \left( i_j-(k-j)-1 \right)}.
\end{equation}
%Now from recall from (\ref{bijection-partition-set}) that the partition
%$\lambda(S)$ associated to $S$ is $\lambda = (\lambda_1,\dots,\lambda_k)$
%where $\lambda_j = i_j-k+j-1$. Therefore
%\begin{equation}
%\label{lambda-wt-q-specialized}
%\specialize_q \wt(S) = \prod_{j=1}^k q^{\lambda_j} = q^{|\lambda(S)|}.
%\end{equation}

\begin{corollary}
\label{specialized-w-weight}
For any permutation $w$ in $\Symm_n$ one has 
$
\specialize_q \wt(w)  = q^{\inv(w)}.
$
\end{corollary}
\begin{proof}
Induct on $n$, with $n=0$ as a trivial base case.
In the inductive step, if $w_1=k+1$ and 
$w = u \cdot a \cdot b$ is the parabolic factorization from
\eqref{parabolic-factorization}, then
\begin{equation}
\label{additivity-of-inv}
\inv(w) = \inv(u) + \inv(a) + \inv(b).
\end{equation}
Note that $\inv(b)=\inv(\hat{b})$.
Also note that in \eqref{Grassmannian-to-set-encoding},
if one has $S(u)=\{i_1 > \cdots > i_k\}$, then 
$
\inv(u)=\sum_{j=1}^k \left( i_j-(k-j)-1 \right)
$
so that \eqref{specialize-set-weight} implies
\begin{equation}
\label{u-and-S-have-same-q-weight}
q^{\inv(u)} = \specialize_q \wt(S(u)).
\end{equation}
Since by definition one has
$$
\wt(w) = \wt(S(u)) \cdot \wt(a) \cdot F^{k+1} \wt(\hat{b})
$$
the assertion of the corollary follows from
\eqref{additivity-of-inv}, \eqref{u-and-S-have-same-q-weight},
together with the inductive hypothesis applied to $a$ and $\hat{b}$.
\end{proof}

%%%%%%%%%%%%%%%%%%%%%%%%%%%%%%%%%%%%%%%%%%%%%%%%%%%%%%%%%%%%%%%%%%%%%%%%
\section{The weight of a permutation, via a search tree}
\label{permutation-tree-section}
%%%%%%%%%%%%%%%%%%%%%%%%%%%%%%%%%%%%%%%%%%%%%%%%%%%%%%%%%%%%%%%%%%%%%%%%

The goal of this section is to encode the recursive nature of the definition
of the weight $\wt(w)$ for a permutation $w$
in a standard combinatorial data structure, an increasing binary search tree.
Once this tree is computed, one no longer needs recursion
to define $\wt(w)$.

\newcommand{\IC}{\operatorname{T}}

\begin{definition} (cf. Stanley \cite[\S 1.3]{Stanley})
  For any word $w=w_1\dots w_m$ without repetition, define recursively its
  \emph{increasing binary tree} $\IC(w)$ as follows:
  \begin{itemize}
  \item if $w$ is empty (i.e. $m=0$), then $\IC(w)$ is the empty binary tree;
  \item else denote by $k$ the index of the smallest letter of $w$. Then
    $\IC(w)$ is the binary tree whose root is labelled $w_k$, whose left
    subtree is $\IC(w_1\dots w_{k-1})$ and whose right subtree is $\IC(w_{k+1}\dots
    w_m)$.
  \end{itemize}
\end{definition}
Now for a given permutation $w$, consider the tree $\IC(w^{-1})$.
For each pair of labeled nodes $(\alpha, \beta)$ such that 
$\alpha$ occurs in the left subtree rooted at $\beta$, 
define a {\it numerator} polynomial
$N(\alpha,\beta)$ and {\it denominator} polynomial $D(\alpha,\beta)$ by
$$
\begin{array}{rll}
   D(\alpha,\beta) &:= x_{w(\beta)-1} + \dots + x_{w(\beta)-\ell} 
                        &= F^{w(\beta)-\ell-1} [\ell] \\
   N(\alpha,\beta) &:= F^{r+1}(D(\alpha,\beta)) 
                        &=F^{w(\beta)+r-\ell} [\ell]
\end{array}
$$
where $\ell:=\ell(\alpha,\beta)$ (resp. $r:=r(\alpha,\beta)$) is the number of nodes
in the left (resp. right) subtree of $\beta$ whose label is larger or equal
(resp. smaller) than $\alpha$. Note that since $\alpha$ is in the 
left subtree of $\beta$, one always has $\ell \geq 1$.

\begin{example}
  \label{example-non-recursive-weight}
  For example, consider the permutation $w = 541736829$. Its inverse is
  $w^{-1}=385216479$. The corresponding increasing tree $\IC(w^{-1})$ is therefore
  \[
  \IC(w^{-1})\ =\
  %% 1
  %% / \
  %% _ 2 4
  %% _/ / \
  %% 3 6 7
  %% \ \
  %% 5 9
  %% /
  %% 8
  \raisebox{1mm}
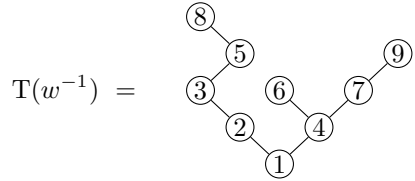
{
  \begin{tikzpicture}[baseline=(current bounding box.center), scale=0.7, inner
    sep=0.3mm, % distance from a label to its surrouding circle
    level distance=7mm % distance between two levels
%    ,edge from parent/.style={<-,draw}
   ]
    \node[circle, draw] {1} [grow'=up] 
        child {node[circle, draw] {2}
           child{node[circle,draw] {3}
             child{ {} edge from parent[draw=none]}
             child{node[circle,draw] {5}
               child{node[circle,draw] {8}}
               child{ {} edge from parent[draw=none]}
             }
           }
           child{ {} edge from parent[draw=none]}
        } 
        child {node[circle, draw] {4}
           child{node[circle,draw] {6}}
           child{node[circle,draw] {7}
             child{ {} edge from parent[draw=none]}
             child{node[circle,draw] {9}}
           }
        } 
    ;
  \end{tikzpicture}}
  \]
and the relevant pairs $(\alpha,\beta)$ and polynomials 
$N(\alpha,\beta), D(\alpha,\beta)$ are as follows:
  \begin{equation*}
    \begin{array}{|c|c|c|c|c|c|c|}\hline
      \alpha & \beta & w(\beta) & \ell & r & N(\alpha,\beta)         & D(\alpha,\beta)\\
             &       &          &      &   &:=F^{r+1} D(\alpha,\beta)& :=F^{w(\beta)-\ell-1}([\ell]) \\\hline\hline
      2 & 1 &  5        & 4 & 0 & x_5+x_4+x_3+x_2 & x_4+x_3+x_2+x_1\\\hline
      3 & 1 &  5        & 3 & 0 & x_5+x_4+x_3     & x_4+x_3+x_2\\\hline
      5 & 1 &  5        & 2 & 1 & x_6+x_5         & x_4+x_3\\\hline
      8 & 1 &  5        & 1 & 3 & x_8             & x_4\\\hline
      3 & 2 &  4        & 3 & 0 & x_4+x_3+x_ 2    & x_3+x_2+x_1\\\hline
      5 & 2 &  4        & 2 & 0 & x_4+x_3         & x_3+x_2\\\hline
      8 & 2 &  4        & 1 & 0 & x_4             & x_3\\\hline
      8 & 5 &  3        & 1 & 0 & x_3             & x_2 \\\hline
      6 & 4 &  7        & 1 & 0 & x_7             & x_6\\\hline
    \end{array}
  \end{equation*}
\end{example}
\begin{proposition}
  For any permutation $w$, the weight of $w$ equals
  \begin{equation}
    \label{equation-non-recursive-weight}
    \wt(w) =
    \prod_{(\alpha,\beta)} \frac{N(\alpha,\beta)}{D(\alpha,\beta)}\,,
  \end{equation}
  where the product is over $(\alpha, \beta)$ with
  $\alpha$ in the left subtree of $\IC(w^{-1})$ rooted at $\beta$.
\end{proposition}
\begin{proof}
  Induct on $n$, with trivial base cases $n =0,1$. 
  In the inductive step, let $L$ and $R$ be the left
  and right subtrees of the root of $\IC(w^{-1})$. 
  Define $a$, $u$ and $\hat b$ as in
  Definition \ref{definition-weight-permutation}. Then
  \begin{equation}
    \wt(w):=  \wt(S(u)) \cdot \wt(a) \cdot F^{k+1}( \wt(\hat{b}) ).
  \end{equation}
  Assume \eqref{equation-non-recursive-weight} holds for $w:=a$ or $w:=\hat b$;
  we wish to prove it holds for $w=u \cdot a \cdot b$.

  The tree $\IC(a^{-1})$ is obtained from $L$ by renumbering the
  labels to $\{1,\dots,k\}$ keeping their relative order. Let
  $(\alpha,\beta)$ be two nodes of $L$ and $(\alpha',\beta')$ their renumbering
  in $\IC(a^{-1})$. It should be clear that 
$$
\begin{aligned}
r(\alpha,\beta)&=r(\alpha',\beta')\,, \\
\ell(\alpha,\beta)&=\ell(\alpha',\beta')\,, \\
w(\beta) &= a(\beta')\,.
\end{aligned}
$$
As a consequence
  \begin{equation}
    \wt(a) = 
    \prod_{(\alpha',\beta')}
    \frac{N(\alpha',\beta')}{D(\alpha',\beta')}
    =
    \prod_{\substack{(\alpha,\beta)\\ \alpha,\beta\in L}}
    \frac{N(\alpha,\beta)}{D(\alpha,\beta)}\,.
  \end{equation}
  Similarly, the values of $r$ and $\ell$ also agree in $\IC({\hat b}^{-1})$ and
  $R$, but the difference is that for two corresponding nodes $\beta\in R$ and
  $\beta'\in \IC({\hat b}^{-1})$, one has $w(\beta) = {\hat b}(\beta') + k
  + 1$. It follows that
  \begin{equation}
    F^{k+1}(\wt(\hat b)) =
    \prod_{(\alpha',\beta')}
    F^{k+1}\left(\frac{N(\alpha',\beta')}{D(\alpha',\beta')}\right) =
    \prod_{\substack{(\alpha,\beta)\\ \alpha,\beta\in R}}
    \frac{N(\alpha,\beta)}{D(\alpha,\beta)}\,.
  \end{equation}
  It therefore remains to show that $\wt(S(u))$ is exactly the product over
  pairs $(\alpha, \beta)$ with $\alpha = 1$. Ordering decreasingly the
  labels $\{\alpha_1>\dots>\alpha_k\}$ of $L$ which are also the elements of
  $S(u)$, one sees that 
$$
\begin{aligned}
\ell(\alpha_j, 1) &= j\,,\\
r(\alpha_j, 1) &=  \alpha_j-1-(k-j)\,.
\end{aligned}
$$
Since $w(1) = k+1$, one has
$$
  \begin{aligned}
    D(\alpha_j, 1) &= F^{k-j}[j]\,, \\
    N(\alpha_j, 1) &= F^{\alpha_j-1}[j]\,.
  \end{aligned}
$$
  Therefore
  \begin{equation}
    \prod_{\alpha\in L}
    \left(\frac{N(\alpha,1)}{D(\alpha,1)}\right) =
    \prod_{j=1}^f \frac{F^{\alpha_j-1}[j]}{F^{k-j}[j]} =
    \frac{\prod_{j=1}^k F^{\alpha_j-1}[j]}{[k]!} =
    \wt(S(u))\,.
  \end{equation}
  This proves that \eqref{equation-non-recursive-weight} holds for $w=u\cdot a
  \cdot b$. 
\end{proof}
\begin{example}
  Continuing Example \ref{example-non-recursive-weight}, one sees that $a =
  4132$ so that $a^{-1} = 2431$ and $b = 57689$ so that $\hat{b} = 2134$ and 
  $\hat{b}^{-1} = 2134$. As a consequence:

 \begin{equation*}
    \IC(a^{-1}) =
 %   \begin{tikzpicture}[baseline=(current bounding box.center), scale=0.7, inner
 %     sep=0.3mm, % distance from a label to its surrouding circle
 %     level distance=7mm, % distance between two levels
 %     edge from parent/.style={<-,draw}]
 %     \node[circle, draw] {1} child {node[circle,
 %       draw] {2} child[missing] {} child {node[circle, draw] {3} child
 %         {node[circle, draw] {4}} child[missing] {} } }
 %     child[missing] {};
 %   \end{tikzpicture}
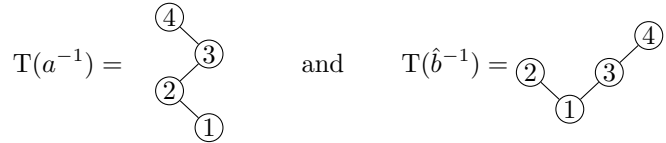
  \begin{tikzpicture}[baseline=(current bounding box.center), scale=0.7, inner
    sep=0.3mm, % distance from a label to its surrouding circle
    level distance=7mm % distance between two levels
%    ,edge from parent/.style={<-,draw}
   ]
    \node[circle, draw] {1}[grow'=up]
           child{node[circle,draw] {2}
             child{ {} edge from parent[draw=none]}
             child{node[circle,draw] {3}
               child{node[circle,draw] {4}}
               child{ {} edge from parent[draw=none]}
             }
           }
           child{ {} edge from parent[draw=none]}
      ;
  \end{tikzpicture}
    \qquad\text{and}\qquad
   \IC(\hat{b}^{-1}) =
  \begin{tikzpicture}[baseline=(current bounding box.center), scale=0.7, inner
    sep=0.3mm, % distance from a label to its surrouding circle
    level distance=7mm % distance between two levels
%    ,edge from parent/.style={<-,draw}
   ]
       \node[circle, draw] {1} [grow'=up]  
           child{node[circle,draw] {2}}
           child{node[circle,draw] {3}
             child{ {} edge from parent[draw=none]}
             child{node[circle,draw] {4}}
           }
    ;
  \end{tikzpicture}
%    \begin{tikzpicture}[baseline=(current bounding box.center), scale=0.7, inner
%      sep=0.3mm, % distance from a label to its surrouding circle
%      level distance=7mm, % distance between two levels
%      edge from parent/.style={<-,draw}]
%      \node[circle, draw] {1} child {node[circle, draw] {2}} child
%      {node[circle, draw] {3} child[missing] {} child {node[circle, draw] {4}}
%      };
%    \end{tikzpicture}
\end{equation*}

\end{example}

\noindent
This gives a different way to view the assertion 
$
\specialize_q \wt(w)  = q^{\inv(w)}.
$
of Corollary~\ref{specialized-w-weight}.

\begin{proof}[Second proof of Corollary~\ref{specialized-w-weight}]
Rephrasing Proposition~\ref{equation-non-recursive-weight} as
$$
\wt(w) = \prod_{(\alpha,\beta)} \frac{F^{r(\alpha,\beta)+1} D(\alpha,\beta)}{D(\alpha,\beta)}
$$
and bearing in mind \eqref{each-Frobenius-is-a-q}, it suffices to check that
  \begin{equation}
  \label{inversions-via-increasing-tree}
    \inv(w) = \sum_{(\alpha,\beta)} \left( r(\alpha,\beta) + 1 \right) \,.
  \end{equation}
  Let $(i<j)$ be an inversion of $w$, meaning that $w_j<w_i$. Looking
  at $w^{-1}$, this means that $j$ occurs to the left of $i$ in the word
  $w^{-1}=(w^{-1}(1), w^{-1}(2), \dots, w^{-1}(n))$.
  There are two possibilities:
  \begin{itemize}
  \item {\sf For all $r$ such that $w_j<r<w_i$ one has $i<w^{-1}(r)$.} \\
     In other words, in $w^{-1}$ all letters between $j$ and
    $i$ are bigger than $i$. By the construction of the tree
    $T=\IC(w^{-1})$, this implies that $j$ lies in the left subtree of $i$.
  \item {\sf There exists an $r$ such that $w_j<r<w_i$ and $w^{-1}(r)<i$.}  \\
    In other words, one can find a letter smaller than $i$ lying
    between $j$ and $i$ in $w^{-1}$. Let $k$ be the minimal such letter:
    \begin{equation}
      k:=\min\{w^{-1}(r)\ \mid\ w_j<r<w_i\}.
    \end{equation}
    By the construction of $T=\IC(w^{-1})$, the letter $k$ is the label of
    the only node $m$ of $T$ such that $j$ and $i$ are in the left and right
    subtrees of $m$. Therefore this $i$ counts for $1$ in $r(\alpha,\beta)$
    where $\alpha:=j$ and $\beta:=k$.
  \end{itemize}
  As a consequence, fixing $\alpha$, the sum $\sum_\beta \left(r(\alpha,\beta)+1 \right)$ 
  is exactly the number of $i<\alpha$ such that $w_i >  w_\alpha$.  This 
  proves \eqref{inversions-via-increasing-tree}.
\end{proof}

%%%%%%%%%%%%%%%%%%%%%%%%%%%%%%%%%
\section{Proof of Theorem~\ref{main-theorem}}
\label{proof-section}
%%%%%%%%%%%%%%%%%%%%%%%%%%%%%%%%%

For a recursively labelled forest $P$ on $\{1,2,\ldots,n\}$, we wish to
prove equality of the two rational functions
\begin{equation}
\label{define-LHS-RHS}
\begin{aligned}
L(P)&:=\sum_{w \in \LLL(P) } \wt(w)\,, \\
H(P)&:=\frac{[n]!}{\prod_{i=1}^n F^{\min(P_{\geq i})-1}[h_i]}\,.
\end{aligned}
\end{equation}
Proceed by induction on the following quantity:  the
sum of $n$ and the number of incomparable pairs $i,j$ in $P$.
In the base case where this quantity is zero, in particular $n=0$,
and the result is trivial. In the inductive step, there are two cases.
\vskip .1in
\noindent
{\sf Case 1.} There exist two elements $i,j$ having
subtrees $P_{\geq i},  P_{\geq j}$ labelled by two contiguous
intervals of integers, say 
$$
\begin{aligned}
P_{\geq i}&=[r+1,r+s]\,,\\
P_{\geq j}&=[r+s+1,r+s+t]\,.
\end{aligned}
$$

In this case, form the poset $P_{i < j}$ by
taking the transitive closure of $P$ and the extra relation
$i<j$.  Defining $P_{j < i}$ similarly, one has the disjoint decomposition
$$
\LLL(P) = \LLL(P_{i < j}) \sqcup \LLL(P_{j < i})
$$
since any $w$ in $\LLL(P)$ either has $i <_w j$ or
$j <_w i$.  Therefore 
$$
L(P) = L(P_{i < j}) + L(P_{j < i})\,,
$$
and hence it remains to show 
\begin{equation}
\label{desired-hook-formula-recursion}
H(P) = H(P_{i < j}) + H(P_{j < i})\,.
\end{equation}
Because $P, P_{i<j}, P_{j<i}$ share the same size $n$,
and share the same label sets on their subtrees
$P_{\geq k}$ for $k \neq i,j$,
the desired equality \eqref{desired-hook-formula-recursion}
is equivalent to checking
$$
\frac{1}{F^r[s] \cdot F^{r+s}[t]} =
\frac{1}{F^{r}[s+t] \cdot F^{r+s}[t]} 
+\frac{1}{F^r[s] \cdot F^{r}[s+t]}\,.
$$
Over a common denominator, this amounts to checking 
$$
F^r[s+t] = F^r[s]+F^{r+s}[t]\,,
$$
which is immediated from the definition \eqref{n-definition} of $[n]$.

\vskip .1in
\noindent
{\sf Case 2.} There are no such pairs of elements $i,j$ as
in Case 1.

  This means that $P$ is a {\it recursively labelled binary tree}, meaning
that it has a minimum element, say $k+1$, and every element $i$ in $P$
is covered by at most one element $j<_\ZZ i$ and at most one
element $j >_\ZZ i$.  In particular, this means that the poset $P_1$
obtained by restricting $P$ to the values $[1,k]$ is again a recursively
labelled binary tree.  Similarly the restriction of $P$ to the values $[k+2,n]$
is obtained from some recursively labelled binary tree $P_2$ on 
values $[1,n-k-1]$ by adding $k+1$ to all of its vertex labels;  
denote this restriction $F^{k+1}(P_2)$.

  One then calculates that
$$
\begin{aligned}
H(P)&= \frac{[n]!}{\prod_{i=1}^n F^{\min(P_{\geq i})-1}[h_i]} \\
&= \frac{F[n-1]!}{\prod_{i \neq k+1} F^{\min(P_{\geq i})-1}[h_i]} \\
&= \frac{F[k]!}{[k]!} \cdot
    F \qtbin{n-1}{k} \cdot
    \frac{[k]!}{\prod_{i=1}^k F^{\min(P_{\geq i})-1}[h_i]}  \cdot
    \frac{F^{k+1}[n-1-k]!}{\prod_{i=k+2}^n F^{\min(P_{\geq i})-1}[h_i]} \\
&= \frac{F[k]!}{[k]!} \cdot
    F \qtbin{n-1}{k} \cdot
    H(P_1) \cdot F^{k+1}H(P_2)\,.
\end{aligned}
$$
It remains to show that $L(P)$ satisfies the same recurrence.
Note that each $w$ in $\LLL(P)$ has $w_1=k+1$, 
because $k+1$ is the minimum element of $P$.  Furthermore,
when one decomposes $w=u \cdot a \cdot b$ as in
the parabolic factorization \eqref{parabolic-factorization} used to define
$\wt(w)$, one finds that $a, \hat{b}$  lie in
$\LLL(P_1), \LLL(P_2)$, respectively.  Conversely, any
such triple $(u,a, \hat{b})$, where $u$ is a shuffle
of the sequences $(1,2,\ldots,k), (k+1,k+2,\ldots,n)$
having $u(1)=k+1$, gives rise to an element $w=u \cdot a \cdot b$ of $\LLL(P)$.
Thus
$$
\begin{aligned}
L(P) &=\sum_{(u,a,\hat{b})} \wt(S(u)) \wt(a) F^{k+1} \wt(\hat{b}) \\
&=\frac{F[k]!}{[k]!} \left( \sum_{\substack{k-\text{subsets }\hat{S}\\ \text{ of }\{1,2,\ldots,n-1\} }} 
     F\left(\wt(\hat{S})\right) \right)
                           \left( \sum_a \wt(a) \right) 
                           \left( \sum_{\hat{b}} F^{k+1}  \wt(\hat{b}) \right)\\
&=\frac{F[k]!}{[k]!} F\qtbin{n-1}{k} \cdot L(P_1) \cdot F^{k+1} L(P_2)
\end{aligned}
$$
using \eqref{w-weight-with-hat-lambda} and \eqref{binomial-as-sum}.

\vskip.2in
Thus in both cases, $L(P)$ and $H(P)$ satisfy the same recurrence, 
concluding the proof of Theorem~\ref{main-theorem}.

%%%%%%%%%%%%%%%%%%%%%%%%%%%%%%%%%%%%%%%%%%%%%%%%%%%%%%%%%
\section{Specializing to the formula of Bj\"orner and Wachs}
\label{corollaries-section}
%%%%%%%%%%%%%%%%%%%%%%%%%%%%%%%%%%%%%%%%%%%%%%%%%%%%%%%%%

\noindent
It is now easy to deduce Bj\"orner and Wachs' identity \eqref{Bjorner-Wachs-equation}
as the $q$-specialization of Theorem~\ref{main-theorem}:
one has from Corollary~\ref{specialized-w-weight} that 
$$
\specialize_q L(P)=\sum_{w \in \LLL(P)}q^{\inv(w)}\,,
$$
while the right side $H(P)$ of Theorem~\ref{main-theorem} has $q$-specialization
$$
\begin{aligned}
\specialize_q H(P)
&= \specialize_q 
\prod_{i=1}^n \frac{F^{i-1}[n-i+1]}{F^{\min(P_{\geq i})-1}[h_i]}\\
&= q^{\sum_{i=1}^n (i-\min(P_{\geq i}))} 
      \prod_{i=1}^n \frac{1-q^{n-i+1}}{1-q^{h_i}} 
\quad \text{ using \eqref{quotient-q-specialized}} \\
&= q^{\inv(P)} \frac{[n]!_q}{\prod_{i=1}^n [h_i]}
\end{aligned}
$$
where the last equality used the following fact:  since
$P$ is a recursively labelled forest, for each $i$, the
quantity $i-\min(P_{\geq i})$ counts the contribution to $\inv(P)$ coming
from the pairs $(i,j)$ where $j$ lies in $P_{\geq i}$.

%%%%%%%%%%%%%%%%%%%%%%%%%%%%%%%%%%%%%%%%%%%%%%%%%%%%%%%%%
\section{Algebra morphisms}
\label{algebra-section}
%%%%%%%%%%%%%%%%%%%%%%%%%%%%%%%%%%%%%%%%%%%%%%%%%%%%%%%%%
Theorem~\ref{main-theorem} has an interesting rephrasing in terms of a
$\QQ$-linear map from the ring of 
{\it free quasisymmetric functions} $\FQSym$ (or {\it Malvenuto-Reutenauer algebra})
into a certain target ring.  We define these objects here.

\begin{definition}
Recall from \cite{MalvenutoReutenauer} that the algebra $\FQSym$ 
has $\QQ$-basis elements
$$
\left\{ \FFF_w : w \in \bigsqcup_{n \geq 0} \Symm_n \right\},
$$ 
with multiplication defined $\QQ$-bilinearly as follows:
for $a, b$ lying in $\Symm_k, \Symm_\ell$ one has
$$
\FFF_a \cdot \FFF_b := 
\sum_{w} \FFF_w
$$
where
$w$ runs through all {\it shuffles} of the words 
$$
\begin{aligned}
a&=(a_1,\ldots,a_k), \text{ and }\\
F^k(b)&:=(b_1+k,\ldots,b_\ell+k)\,.
\end{aligned}
$$
\end{definition}
\noindent
One of the original motivations for introducing the 
ring $\FQSym$ is the following.  Define for
each poset $P$ the element
\begin{equation}
\FFF_P := \sum_{w \in \LLL(P)} \FFF_w
\end{equation}
in $\FQSym$.  Then for two posets $P, Q$ on elements $\{1,2,\ldots,k\},
\{1,2,\ldots,\ell\}$, respectively, one has in $\FQSym$ that
\begin{equation}
\label{FQSym-motivation}
\FFF_P \cdot \FFF_Q = \FFF_{P \sqcup F^k(Q)}\,,
\end{equation}
where $P \sqcup F^k(Q)$ denotes the poset on $\{1,2,\ldots,k+\ell\}$ which
is the disjoint union of $P$ with the
poset $F^k(Q)$ on $\{k+1,k+2,\ldots,k+\ell\}$ obtained by 
adding $k$ to each label in $Q$.

\begin{definition}
Let the semigroup $\NN=\{1,F,F^2,\ldots\}$ 
act on the rational functions $\QQ(\xx)=\QQ(x_1,x_2,\ldots)$ 
via the Frobenius map as before:
$F(x_i)=x_{i+1}$.  Then define the {\it skew semigroup algebra} $\QQ(\xx)\#\NN$
to be the free $\QQ(\xx)$-module on basis $\{1,u,u^2,\ldots\}$,
with multiplication defined $\QQ$-linearly by 
$$
f(\xx) u^k \cdot g(\xx) u^\ell=
\left(f(\xx) F^k(g(\xx))\right) u^{k+\ell}\,.
$$
\end{definition}
\noindent
One of our motivations for introducing $\QQ(\xx)\#\NN$ is that,
in addition to its $\QQ(\xx)$-basis $\{1,u,u^2,\ldots\}$, it
also has a $\QQ(\xx)$-basis of {\it divided powers} 
$\{1,u^{(1)},u^{(2)},\ldots\}$, where 
$$
u^{(n)}:= \frac{1}{[n]!} u^n,
$$
and this basis has our binomial coefficients as multiplicative 
structure constants:
\begin{equation}
\label{binomial-convolution}
u^{(k)} \cdot u^{(\ell)} = \qtbin{k+\ell}{k} u^{(k+\ell)}\,.
\end{equation}

\begin{definition}
Define the $\QQ$-linear map 
$$
\begin{array}{rcl}
\FQSym & \overset{\phi_{\inv}}{\longrightarrow} & \QQ(\xx)\#\NN \\
\FFF_w&\longmapsto &\frac{\wt(w)}{[n]!} u^n = \wt(w) \cdot u^{(n)}
\end{array}
$$
for $w$ in $\Symm_n$.  Note that
\begin{equation}
\label{phi-inv-on-posets}
\phi_{\inv}( \FFF_P ) = L(P) \cdot u^{(n)}\,.
\end{equation}
\end{definition}

\noindent
This $\QQ$-linear map $\phi_{\inv}$ turns
out {\it not} to be an algebra morphism. 
E.g., one can check via explicit computations that
$$
\begin{aligned}
\phi_{\inv}(\FFF_1 \cdot \FFF_{213})
&=\phi_{\inv}(\FFF_{1324})+\phi_{\inv}(\FFF_{3124})
  +\phi_{\inv}(\FFF_{3214})+\phi_{\inv}(\FFF_{3241}) \\
&\neq \phi_{\inv}(\FFF_1) \cdot \phi_{\inv}(\FFF_{213})\,.
\end{aligned}
$$

\noindent
However, the import of Theorem~\ref{main-theorem} is that
$\phi_{\inv}$ {\it becomes} an algebra morphism when restricted
to an appropriate subalgebra of $\FQSym$.

\begin{definition}
Recall from \cite{LodayRonco} that the 
{\it Loday-Ronco algebra of binary trees} $\PBT$ can be
defined as the subalgebra of $\FQSym$ spanned by all $\{\FFF_P\}$
as $P$ runs through all recursively labelled forests.
\end{definition}

\begin{proposition}
When restricted from $\FQSym$ to $\PBT$, the
map $\phi_{\inv}$ becomes an algebra homomorphism 
$\PBT \overset{\phi_{\inv}}{\longrightarrow} \QQ(\xx)\#\NN$.
\end{proposition}
\begin{proof}
It is easy to check that the product formula $H(P)$ 
defined in \eqref{define-LHS-RHS} for a recursively labelled forest $P$ 
satisfies
\begin{equation}
\label{hook-product-disconnected}
H(P \sqcup F^k Q) =
\qtbin{k+\ell}{k} H(P) \cdot F^k H(Q)\,.
\end{equation}
Hence for recursively labelled forests $P,Q$ of sizes $k, \ell$, one has
$$
\begin{array}{rll}
\phi_{\inv}(\FFF_P \cdot \FFF_Q)
& = \phi_{\inv} \left(\FFF_{P \sqcup F^k Q} \right) 
      &\text{ by \eqref{FQSym-motivation}}\\
& = L(P \sqcup F^k Q) \cdot u^{(k+\ell)} 
      &\text{ by \eqref{phi-inv-on-posets}}\\
& = H(P \sqcup F^k Q) \cdot u^{(k+\ell)} 
      &\text{ by Theorem \ref{main-theorem}}\\
& = \qtbin{k+\ell}{k} H(P) \cdot F^k H(Q) \cdot u^{(k+\ell)} 
      &\text{ by \eqref{hook-product-disconnected}}\\
& = H(P) u^{(k)} \cdot H(Q) u^{(\ell)} 
      &\text{ by \eqref{binomial-convolution}}\\
& = L(P) u^{(k)} \cdot L(Q) u^{(\ell)}       
      &\text{ by Theorem \ref{main-theorem}}\\
& = \phi_{\inv}(\FFF_P) \cdot \phi_{\inv}(\FFF_Q) 
      &\text{ by  \eqref{phi-inv-on-posets}} \, .
\end{array}
$$
\end{proof}

\begin{remark}
\label{maj-formula-remark}
This twisted semigroup algebra $\QQ(\xx)\#\NN$ also appears implicitly
in the theory of $P$-partitions, as the target of 
a {\it different} map $\phi_{\maj}: \FQSym \rightarrow \QQ\#\NN$,
which {\it is} an algebra morphism.
This is related to a recent multivariate generalization of
Bj\"orner and Wachs' {\it other} ``maj'' $q$-hook formula
for forests.  We describe both connections briefly here.

For a poset $P$ on $\{1,2,\ldots,n\}$, a {\it $P$-partition} (see
\cite[\S 4.5 and 7.19]{Stanley}) is a 
weakly order-reversing function $f:P \rightarrow \NN$ 
(meaning $i<_Pj$ implies $f(i) \geq f(j)$) which is strictly
decreasing along {\it descent} covering relations: whenever $j$ covers $i$ 
in $P$ and $i >_\ZZ j$ 
then $f(i) > f(j)$.  Define their generating function
$
\gamma(P,\xx):=\sum_{f} \xx^f
$
where here $f$ runs over all $P$-partitions, and
$\xx^f:=x_1^{f(1)} \cdots x_n^{f(n)}.$ 
The relevant algebra morphism is defined $\QQ$-bilinearly as follows:
$$
\begin{array}{rcl}
\FQSym &\overset{\phi_{\maj}}{\longrightarrow} &\QQ\#\NN  \\
\FFF_w &\longmapsto & \gamma(w,\xx) \cdot u^n\,.
\end{array}
$$
The main proposition on $P$-partitions \cite[Theorem 4.54]{Stanley} asserts
that 
\begin{equation}
\label{main-P-partition-prop}
\gamma(P,\xx) =
\sum_{w \in \LLL(P) } \gamma(w,\xx) 
\end{equation}
or equivalently, 
$$
\phi_{\maj}(\FFF_P) = \gamma(P,\xx) u^n\,.
$$
This then shows that $\phi_{\maj}$ is an
algebra morphism, since for any posets
$P,Q$ on $[1,k]$ and $[1,\ell]$, one has
$$
\begin{aligned}
\phi_{\maj}( \FFF_P \cdot \FFF_Q)
&= \phi_{\maj}( \FFF_{P \sqcup F^k(Q)} ) \\
&= \gamma( P \sqcup F^k(Q) , \xx) u^{k+\ell}\\
&= \gamma(P,\xx) \cdot F^k(\gamma(Q,\xx)) u^{k+\ell}\\
&= \gamma(P,\xx)u^k \cdot \gamma(Q,\xx)u^\ell\\
&=\phi_{\maj}( \FFF_P) \cdot \phi_{\maj}(\FFF_Q)\,.
\end{aligned}
$$

The Bj{\"o}rner-Wachs {\it maj} formula arises
when $P$ is a {\it dual forest}, that is, 
every element $i$ in $P$ is covered by at most one
other element $j$; say that $i$ is a {\it descent} of $P$ if in addition
$i >_\ZZ j$.  Let $\Des(P)$ denote the set of descents of $P$,
and $\maj(P):=\sum_{i \in \Des(P)} i$.  In particular, permutations
$w=(w_1,\ldots,w_n)$ considered as linear orders are
dual forests, and for them one has $\maj(w)=\sum_{i: w_i > w_{i+1}} i$.
For any dual forest $P$, note that the subtree rooted at $i$ is $P_{\leq i}$,
and again denote its cardinality by $h_i$.  The Bj{\"o}rner-Wachs {\it maj} formula
asserts the following.
\vskip.1in
\noindent
{\bf Theorem.} (\cite[Theorem 1.2]{BjornerWachs})
{\it
Any dual forest $P$ on $\{1,2,\ldots,n\}$ has
\begin{equation}
\label{Bjorner-Wachs-maj-formula}
\sum_{w \in \LLL(P) } q^{\maj(w)} = q^{\maj(P)} \frac{[n]!_q}{\prod_{i=1}^n [h_i]_q}\,.
\end{equation}
}
\vskip.1in
\noindent
The following generalization was observed recently in \cite{BoussicaultFerayR}:
\vskip.1in
\noindent
{\bf Theorem.} 
{\it
For any dual forest $P$ on $\{1,2,\ldots,n\}$, one has
\begin{equation}
\label{forest-maj-product}
\gamma(P,\xx):=
\frac{\prod_{i \in \Des(P)} \xx_{P_{\leq i}}}{\prod_{i=1}^n (1-\xx_{P_{\leq i}})}
\end{equation}
where $\xx_S:=\prod_{j \in S} x_j$, so
that \eqref{main-P-partition-prop}
becomes
\begin{equation}
\label{forest-maj-identity}
\sum_{w \in \LLL(P)} 
\left( 
\frac{\prod_{i \in \Des(w)} x_{w_1} \cdots x_{w_i}}
          {\prod_{i=1}^n (1-x_{w_1} \cdots x_{w_i})}
\right)
=\frac{\prod_{i \in \Des(P)}x_{P_{\leq i}}}{\prod_{i=1}^n (1-x_{P_{\leq i}})}\,.
\end{equation}
}
\vskip.1in
\noindent
The Bj\"orner-Wachs $\maj$ formula is immediate
upon specializing $x_i=q$ in \eqref{forest-maj-identity}:
$$
\sum_{w \in \LLL(P)} \frac{q^{\maj(w)}}{(1-q)(1-q^2) \cdots(1-q^n)}
=\frac{q^{\maj(P)}}{\prod_{i=1}^n (1-q^{h_i})}\,.
$$
\end{remark}

\begin{remark}
The maps $\phi_{\inv}, \phi_{\maj}: \FQSym \rightarrow \QQ(\xx)\#\NN$
are reminiscent of the formalism of {\it moulds} discussed by Chapoton,
Hivert, Novelli and Thibon \cite{CHNT}, but we have not yet found a deeper
connection.

One might also hope that 
the $(q,t)$-specializations $\specialize_{q,t} L(P)$ 
for recursively labelled binary trees $P$ can be given a 
representation-theoretic interpretation, similar to the discussion
in Section~\ref{invariant-theory-section}, but related to
$q$-analogues of the indecomposable projective modules
for the algebras whose existence is conjectured by
Hivert, Novelli and Thibon in \cite[\S 5.2]{HivertNovelliThibon}.
At the moment this is purely speculative.
\end{remark}

%%%%%%%%%%%%%%%%%%%%%%%%%%%%%%%%%%%%%%%%%%%%%%%%%%%%%%%%%
% \section*{Acknowledgements}
%%%%%%%%%%%%%%%%%%%%%%%%%%%%%%%%%%%%%%%%%%%%%%%%%%%%%%%%%

\end{document}